%% file: uniqueness.tex
\documentclass[notitlepage,twoside,a4paper]{amsart}
\usepackage{mathrsfs}
\usepackage{amsmath,amssymb,enumerate}
\usepackage{epsfig,fancyhdr,color}
\usepackage{epstopdf}
\usepackage{amssymb}
\usepackage{amsmath,amsthm}
\usepackage{mathtools}
\usepackage{latexsym}
\usepackage{amscd}
\usepackage{psfrag}
\usepackage{graphicx}
\usepackage{epsf}
\usepackage[latin1]{inputenc}
\usepackage[all]{xy}
\usepackage{tikz}
\usepackage{prettyref}
\usepackage{subfigure}
\usepackage{float}
\usepackage{color}
\usepackage{array}
\usepackage{cite}
\usepackage{hyperref}

\usepackage[totalwidth=16cm,totalheight=24cm]{geometry}

\newcommand{\III}{I\hspace{-0.1cm}I\hspace{-0.1cm}I}

\DeclareMathOperator{\area}{Area}

\DeclareMathOperator{\Imagin}{Im}
\DeclareMathOperator{\hol}{hol}

\newtheorem{theorem}{\rm\bf Theorem}[section]

\newtheorem{lemma}[theorem]{\rm\bf Lemma}

\newtheorem{definition}[theorem]{\rm\bf Definition}

\newtheoremstyle{named}{}{}{\itshape}{}{\bfseries}{.}{.5em}{#1 \thmnote{#3}}
\theoremstyle{named}

\newcommand{\C}{{\mathbb{C}}}

\newcommand{\N}{{\mathbb{N}}}

\newcommand{\HH}{{\mathbb{H}}}
\newcommand{\R}{{\mathbb{R}}}

\newcommand{\Z}{{\mathbb{Z}}}

\newcommand{\cC}{{\mathcal{C}}}
\newcommand{\cCC}{{\mathcal{CC}}}

\newcommand{\cH}{{\mathcal{H}}}

\newcommand{\cR}{{\mathcal{R}}}
\newcommand{\cT}{{\mathcal{T}}}
\newcommand{\chT}{\widehat{\mathcal{T}}}
\newcommand{\cML}{{\mathcal{ML}}}
\newcommand{\cCP}{{\mathcal{CP}}}

\newcommand{\tE}{\tilde E}
\newcommand{\tS}{\tilde S}

\newcommand{\epic}{\twoheadrightarrow}
\newcommand{\monic}{\xhookrightarrow{}}

\newcommand{\restr}[1]{|_{#1}}

\newcommand{\cMLrea}{\cML^{\rm{realizable}}_{\partial M}}
\newcounter{notes}%

\def\interieur#1{\mathord{\mathop{\kern{0pt #1}}\limits^\circ}}

\title[Pleating lamination]{Convex co-compact hyperbolic manifolds are determined by their pleating lamination}

\author{Bruno Dular}
\address{Bruno Dular:
University of Luxembourg, FSTM, Department of Mathematics, 
Maison du nombre, 6 avenue de la Fonte,
L-4364 Esch-sur-Alzette, Luxembourg}
\email{bruno.dular@uni.lu}

\author{Jean-Marc Schlenker}
\address{Jean-Marc Schlenker:
University of Luxembourg, FSTM, Department of Mathematics, 
Maison du nombre, 6 avenue de la Fonte,
L-4364 Esch-sur-Alzette, Luxembourg}
\email{jean-marc.schlenker@uni.lu}
\thanks{JMS was partially supported by FNR project O20/14766753.}

\date{v2, \today}

\begin{document}

\begin{abstract}
  Convex co-compact 3-dimensional hyperbolic manifolds are uniquely determined by the pleating measured lamination on the boundary of their convex core.
\end{abstract}

\maketitle

\tableofcontents

\section{Introduction}

\subsection{Background}

Let $M$ be a 3-dimensional convex co-compact hyperbolic manifold. Then $M$ contains a smallest non-empty geodesically convex subset, called its convex core and denoted here by $C(M)$. Thurston \cite[Chapter 8]{thurston-notes} noticed that the boundary of $C(M)$ is homeomorphic to $\partial_\infty M$, and is a disjoint union of locally convex pleated surfaces. As a consequence, its induced metric $m$ is hyperbolic -- of constant curvature $-1$ -- while its pleating is described by a measured lamination $l$ which is geodesic for $m$, called its {\em pleating}  or {\em bending} lamination. Thurston conjectured that $l$ (resp. $m$) uniquely determines $M$. The main goal here is to prove that, indeed, $l$ uniquely determines $M$.

Those two conjectures are relevant for 3-dimensional hyperbolic geometry, but also for Teichm\"uller theory. If true, they would mean that convex co-compact hyperbolic manifold provide a bridge between the ``hyperbolic'' Teichm\"uller theory on the boundary of the convex core -- where the data is in terms of hyperbolic metrics, measured laminations, etc -- and the ``complex'' Teichm\"uller theory on the boundary at infinity, which involves complex structures, holomorphic quadratic differentials, etc. (See e.g. \cite{volumes} for an extension of this point of view.)

The measured laminations on $\partial \bar M$ that can be realized as the pleating lamination of the convex core of a convex co-compact hyperbolic metric were determined by Bonahon and Otal \cite{bonahon-otal}. Their result was extended by Lecuire \cite{lecuire} to manifolds with compressible boundary.

\begin{theorem}[Bonahon--Otal, Lecuire]\label{tm:bonahon-otal}
  Let $\bar M$ be a compact 3-manifold with non-empty boundary, with all boundary components of genus at least 2, and such that the interior $M$ of $\bar M$ admits a complete hyperbolic metric. Let $l$ be a measured lamination on $\partial \bar M$. Assume that:
  \begin{enumerate}
    \item Each closed leaf of $l$ has weight less than $\pi$,
    \item For each essential disk $D$ in $\bar M$, $i(l,\partial D)>2\pi$,
    \item There exists $\eta >0$ such that $i(\partial A,l)\geq\eta$ for each essential annulus $A$ in $\bar M$.
  \end{enumerate}
  Then there exists a non-Fuchsian, convex co-compact hyperbolic metric on $M$ such that $l$ is the measured pleating lamination of the boundary of the convex core of $M$.

  In the quasifuchsian case, the second condition is vacuous and the third condition means that the pair of laminations on the two boundary components of $M$ is filling.
\end{theorem}

This theorem is actually stated in \cite{bonahon-otal,lecuire} in the more general case of geometrically finite hyperbolic manifolds, where closed leaves of weight $\pi$ can appear. Bonahon and Otal \cite{bonahon-otal} also showed that {\em rational} laminations --- measured laminations with support a disjoint union of closed curves --- can be uniquely realized. Moreover, Bonahon \cite{bonahon-almost} showed that any small enough pair of filling measured laminations can be realized uniquely as the measured pleating lamination of the boundary of the convex core of a quasifuchsian manifold close to the Fuchsian locus. For quasifuchsian manifolds over the once punctured torus, uniqueness of the realization was proved by Series \cite{series2006thurstons}.

\begin{definition}
  We denote by  $\cML^{\rm{realizable}}_{\partial M}$ the space of measured laminations on $\partial \bar M$ which satisfy the hypothesis of Theorem \ref{tm:bonahon-otal}.
\end{definition}

\subsection{Result}

The main result here is the following statement, known as Thurston's bending conjecture.

\begin{theorem}\label{tm:main}
  Under the hypothesis of Theorem \ref{tm:bonahon-otal}, the hyperbolic structure on $M$ is uniquely determined by the bending lamination $l$ on the boundary of its convex core.
\end{theorem}

In other terms, each $l\in \cMLrea$ is the bending lamination on the boundary of the convex core for a unique convex co-compact hyperbolic metric on $M$.

The strategy of the proof might also extend to geometrically finite, rather than convex co-compact, hyperbolic manifolds. However, several key ingredients of the proof currently lack a generalization to that case.

\subsection{Main ideas of the proof}

We consider a compact 3-manifold $\bar M$ whose interior $M$ admits a complete hyperbolic metric, with all boundary components of genus at least $2$. We denote by $\cCC_M$ the space of convex co-compact hyperbolic metrics on $M$, considered up to isotopy\footnote{See Section 5.3 in \cite{matsuzaki1998hyperbolic} for a detailed presentation of the distinction between the \emph{isotopy} and \emph{homotopy} deformation spaces.}.

We denote by $\cT_{\partial M}$ the Teichm\"uller space of $\partial \bar M$, and by $\cML_{\partial M}$ the space of measured laminations on $\partial \bar M$. More generally, when $S$ is a closed surface (which typically will be a boundary component of $M$) we denote by $\cT_S$ and by $\cML_S$, respectively, the Teichm\"uller space of $S$ and the space of measured laminations on $S$. For $K\in [-1,0)$, we denote by $\cT_S^K$ the space of metrics of constant curvature $K$ on $S$, which is homeomorphic to $\cT_S$ through rescaling.

Recall that by the celebrated Ahlfors-Bers Theorem (see e.g. \cite[\S 5.1, 5.2]{marden:hyperbolic} or \cite[\S 5.3]{matsuzaki1998hyperbolic}), the map sending a convex co-compact hyperbolic structure $g\in \cCC_M$ to its conformal structure at infinity $c\in\cT_{\partial M}$ is a homeomorphism from $\cCC_M$ to $\cT_{\partial M}$. 

Theorem \ref{tm:main} is equivalent to the statement that the map $\psi$ defined below is injective (outside of the Fuchsian locus).

\begin{definition}
  Let $\psi:\cCC_M\to \cML_{\partial M}$ be the map sending a convex co-compact hyperbolic metric $g\in \cCC_M$ to the measured bending lamination on the boundary of its convex core.
\end{definition}

This map is known to be continuous \cite{keen1995continuity}, \emph{tangentiable} \cite{bonahon1998variations} and surjective onto $\cMLrea$ by Theorem \ref{tm:bonahon-otal} \cite{bonahonotal2004laminations}.

Theorem \ref{tm:main} follows from the two following lemmas on the fibres of $\psi$.
\begin{lemma} \label{lm:contractible}
  For any $l\in \cMLrea$, the fibre $\psi^{-1}(\{l\})$ is contractible.
\end{lemma}

\begin{lemma} \label{lm:non-contractible}
  For any $l\in\cMLrea$, the fibre $\psi^{-1}(\{l\})$ is a compact analytic subset of $\cCC_M$. Therefore it is non-contractible, unless it is a point.
\end{lemma}

\subsubsection{Contractibility of the fibers}

To prove Lemma \ref{lm:contractible}, we will show that $\psi$ is a limit of homeomorphisms. 

\begin{definition}
  Let $K\in (-1,0)$. We denote by $\psi_K: \cCC_M\to \cT^{K^*}_{\partial M}$ the map sending a convex co-compact hyperbolic metric $g$ on $M$ to the third fundamental form on the closed surface $S_K$ of constant curvature $K$ in $M$.
\end{definition}

Here $S_K$ is in general non-connected, and has one connected component for each boundary component of $M$ -- it follows from a result of Labourie (see Theorem \ref{tm:labourie} below) that this $K$-surface is well-defined. The third fundamental for $\III_K$ on $S_K$ then has constant curvature $K^*=K/(K+1)$, see \cite[Prop. 2.3.2]{L5}. It follows from \cite[Theorem 0.2]{hmcb} that for all $K\in (-1,0)$, $\psi_K$ is a homeomorphism onto its image.

\begin{lemma} \label{lm:convergence}
  As $K\to -1$, $\psi_K\to \psi$ pointwise in the marked length spectrum topology. Moreover, for any closed curve $c$ on $\partial \bar M$, $L_{\psi_K(-)}(c)\to i(\psi(-),c)$ uniformly on compact subsets of $\cCC_M$. 
\end{lemma}

In other words, for any closed curve $c$ on $\partial \bar M$ and any $g\in \cCC_M$,
$$ \lim_{K\to -1} L_{\psi_K(g)}(c) \to i(\psi(g), c)~, $$
where $L_{\psi_K(g)}(c)$ denotes the length of the geodesic representative of $c$ in the constant curvature metric $\psi_K(g)$ and $i(\psi(g), c)$ denotes the intersection between the measured lamination $\psi(g)$ and $c$.

Since $\psi$ is a limit of a sequence of homeomorphisms, we can use results from the decomposition theory of manifolds \cite{finney1967pseudo,daverman1986decompositions} to conclude that its fibres are contractible, see Section \ref{ssc:pseudoisotopy}.

\subsubsection{Non-contractibility of the fibers}

To prove that the inverse image by $\psi$ of a point in $\cML_{\partial M}$ is either a point or non-contractible, we will use the following statement.

\begin{lemma} \label{lm:analytic}
  Let $l\in \cML_{\partial M}$ satisfying the hypothesis of Theorem \ref{tm:bonahon-otal}. Then $\psi^{-1}(\{ l\})$ is a real-analytic subset in $\cCC_M$.  
\end{lemma}

Here by ``real-analytic subset'' we mean that, for each $l\in \cML_{\partial M}$ satisfying the hypothesis of Theorem \ref{tm:bonahon-otal}, the subset $\psi^{-1}(\{l\})\subset\cCC_M$ is locally the vanishing set of a real-analytic function valued in a real vector space. The proof of this statement uses Bonahon's shear-bend coordinates on pleated surfaces in the ends of $M$, see Section \ref{sc:analyticity}.

The proof of Lemma \ref{lm:non-contractible} then follows from the fact that compact real analytic spaces have a \emph{fundamental homology class modulo $2$} \cite{borelhaefliger1961classe}. An accessible proof of this fact follows from a result of Sullivan \cite{sullivan:combinatorial} describing the local topology of real-analytic spaces, see Section \ref{sc:non-contractibility}.

\subsection{Further questions}

We do not consider here the ``dual'' question: whether convex co-compact hyperbolic manifolds are uniquely determined by the induced metric on the boundary of their convex core. Our arguments do not, at this point, allow us to conclude that convex co-compact hyperbolic manifolds are {\em infinitesimally rigid} with respect to the measured bending lamination on the boundary of their convex core, that is, that any first-order deformation of a convex co-compact hyperbolic manifold induces a non-zero first-order variation of the bending lamination on the boundary of its convex core. Bonahon \cite{bonahon-toulouse} noticed that this infinitesimal rigidity property is {\em equivalent} to the infinitesimal rigidity with respect to the induced metric on the boundary of the convex core, a statement which would presumably lead to the proof, using existing tools, that convex  co-compact manifolds are uniquely determined by the induced metric on the boundary of their convex core.

Theorem \ref{tm:main} can be considered as a special case of a more general question: whether a convex hyperbolic manifold with boundary can be described in terms of geometric data on its boundary. Two types of data can be considered: the first is the induced metric, considered for polyhedra in $\HH^3$ by Alexandrov \cite{alex}, for convex domains with smooth boundary by Alexandrov and Pogorelov \cite{Po}, for  convex domains in hyperbolic manifolds in \cite{L4,hmcb}. Dually, one can consider a ``bending'' data: for ideal polyhedra in $\HH^3$ it describes the dihedral angles \cite{Andreev-ideal,rivin-annals}, for compact polyhedra their ``dual metric'' \cite{HR}, and for convex domains with smooth boundary in hyperbolic manifolds it corresponds to their third fundamental form \cite{hmcb}, while for convex cores of hyperbolic manifolds it corresponds to the measured bending lamination. Recent results deal with the induced metric \cite{slutskiy:compact,prosanov:polyhedral} and dihedral angles or dual metrics \cite{ideal,prosanov:dual} of convex domains with polyhedral boundary in hyperbolic manifolds. One can also prescribe the induced metric on some boundary components of a convex co-compact manifold and the third fundamental form \cite{short-weyl} or the measured lamination \cite{mesbah:induced} on another. Other results describe the induced metric or third fundamental forms of smooth or polyhedral surfaces in $\HH^3$ invariant under a cocompact action of a Fuchsian group, see e.g. \cite{iie,fillastre2,fillastre3}. 

The problem of prescribing the induced metric (or third fundamental form) of the boundary of a convex hyperbolic manifold can be considered in a ``universal'' setting, where a mostly conjectural picture emerges, see \cite{convexhull,weylgen,weylsurvey}.

The analog of Thurston's question on the bending laminations on the boundary of the convex core was asked by Mess \cite{mess,mess-notes} for ``quasifuchsian'' AdS spacetime. A partial answer was given near the Fuchsian locus \cite{earthquakes}, but neither existence nor uniqueness is known in general, even for rational laminations.

\subsection*{Acknowledgements} We would like to thank Francesco Bonsante, Cyril Lecuire, Filippo Mazzoli and Roman Prosanov for useful discussions related to this paper. The second-named author would like to thank IHES and MPIM Bonn, where part of this work was completed.

\section{Background material}

\subsection{Convex co-compact hyperbolic metrics}

We consider in this paper convex co-compact hyperbolic manifolds, which can be defined as follows: if $M$ is the interior of a compact manifold with boundary, a convex co-compact hyperbolic metric on $M$ is a complete hyperbolic metric $g$ on $M$ such that $(M,g)$ contains a non-empty, compact, geodesically convex subset. (We say that $K\subset M$ is geodesically convex if any geodesic segment in $M$ with endpoints in $K$ is contained in $K$). We always assume $M$ to be oriented.

The space of isotopy classes of convex co-compact hyperbolic metrics on $M$ is denoted by $\cCC_M$. Equivalently, $\cCC_M$ can be described as the space of isotopy classes of \emph{quasi-isometric deformations} (or \emph{quasi-conformal deformations}) of a fixed convex co-compact hyperbolic metric on $M$. The fact that the two definitions coincide follows from Marden's isomorphism theorem \cite[Theorem 8.1]{marden1974geometry}. See also \cite[\S 7.2-3]{canary2004homotopy} and \cite[\S 5.3]{matsuzaki1998hyperbolic}. We will switch between the two descriptions of $\cCC_M$ when convenient.

A convex co-compact hyperbolic manifold $(M,g)$ contains a unique smallest, non-empty, geodesically convex subset $C(M)$, called its convex core. Except in special situations (for Fuchsian manifolds, that is, when $C(M)$ is a totally geodesic surface), $C(M)$ is a three-dimensional domain, with boundary a disjoint union of locally convex, pleated surfaces. Each boundary component of $C(M)$ is then equipped with a hyperbolic metric -- induced by the ambiant hyperbolic metric -- and with a measured lamination, measuring its ``pleating''. See \cite[\S 2]{bonahon-otal}, \cite[Chapter 8]{thurston-notes}.

\subsection{Constant curvature surfaces in hyperbolic ends}

We consider again a convex co-compact hyperbolic manifold $(M,g)$. Since $C(M)$ has locally convex boundary, the exponential Gauss map defines a homeomorphism between the unit normal bundle of $C(M)$ -- the set of unit vectors which are normals of oriented support planes of $C(M)$ at their intersection with $C(M)$ -- and $\partial_\infty M$. Moreover, the geodesic rays defined by those unit normal vectors foliate $M\setminus C(M)$.

It follows that $M\setminus C(M)$ is a disjoint union of ``hyperbolic ends'': non-complete hyperbolic manifolds homeomorphic to $\partial_iM\times [0,+\infty)$, where the $\partial_iM$, $1\leq i\leq n$, are the connected components of $\partial \bar M$, complete on the side corresponding to $+\infty$, and with boundary a concave pleated surface on the side corresponding to $0$. We will denote the boundary at infinity of $E$ by $\partial_\infty E$, and the concave pleated boundary by $\partial_0E$. 

Labourie \cite{L6} proved that, given such a hyperbolic end $E$, it admits a unique foliation by closed surfaces of constant curvature $K$.

\begin{theorem}[Labourie] \label{tm:labourie}
  Let $E$ be a hyperbolic end. It admits a unique foliation by closed surfaces of constant curvature $K$, with $K$ varying from $-1$ close to the concave pleated boundary, to $0$ near the boundary at infinity.
\end{theorem}

As already mentioned above, the third fundamental form on a surface $S_K$ of constant curvature $K$ in $E$ is a metric of constant curvature $K^*=K/(K+1)$. As a consequence, for each $K\in (-1,0)$, each convex co-compact hyperbolic metric $g$ on $M$ determines a metric of constant curvature $K$ on $\partial \bar M$ -- the induced metric on the surfaces of constant curvature $K$ in the ends of $M$ -- as well as a metric of constant curvature $K^*$ -- the third fundamental form on those surfaces. When $M$ has compressible boundary, the third fundamental form has the additional property that the length of compressible curves is strictly larger than $2\pi$. Denote by $\chT^{K^*}_{\partial M}\subseteq\cT^{K^*}_{\partial M}$ the open domain consisting of metrics having that property.

We denote by $\phi_K:\cCC_M\to \cT^{K}_{\partial M}$ and by $\psi_K:\cCC_M\to \cT^{K^*}_{\partial M}$ the maps obtained in this manner. It follows from the main results of \cite{hmcb} that $\phi_K$ is a homeomorphism (Theorem 0.1 in \cite{hmcb}, together with \cite{L4}) and that $\psi_K$ is a homeomorphism onto its image $\chT^{K^*}_{\partial M}$ \cite[Theorem 0.2]{hmcb}.

\subsection{Duality between hyperbolic ends and de Sitter spacetimes}
\label{ssc:duality}

We also need the well-known polar duality between $\HH^3$ and the de Sitter space $dS^3$, as well as between hyperbolic ends and globally hyperbolic maximal compact de Sitter spacetimes. We recall the main properties here, the reader can consult e.g. \cite{RH,shu,FS19} for more details.

Recall that the hyperbolic space $\HH^3$ can be defined as a connected component of a quadric in the 4-dimensional Minkowski space:
$$ \HH^3 = \{ x\in \R^{3,1}~|~\langle x,x\rangle=-1 \text{ and } x_0>0\}~. $$
In the same manner, one can define the de Sitter space $dS^3$ as:
$$ dS^3 = \{ x\in \R^{3,1}~|~\langle x,x\rangle=1 \}~, $$
equipped with its induced metric. It is a geodesically complete, simply connected Lorentzian space of constant curvature $1$.

Given an oriented totally geodesic plane $P\subset \HH^3$, it is the intersection with $\HH^3$ of a time-like linear hyperplane $H\subset \R^{3,1}$. The unit vector $n^*$ positively orthogonal to $H$ in $\R^{3,1}$ is then a point in $dS^3$. Conversely, given any (unoriented) space-like plane $P^*$ in $dS^3$, it is the intersection with $dS^3$ of a unique space-like linear hyperplane $H^*\subset \R^{3,1}$. The future-oriented unit normal to $H^*$ is then a point $n\in \HH^3$.

It follows from this definition that given two oriented totally geodesic planes $P_1$ and $P_2$ in $\HH^3$, with dual points $P_1^*, P_2^*\in dS^3$:
\begin{itemize}
\item $P_1^*$ and $P_2^*$ are connected by a space-like geodesic segment $s$ if and only if $P_1$ and $P_2$ intersect, and the length of $s$ is then the intersection angle between $P_1$ and $P_2$.
\item $P_1^*$ and $P_2^*$ are connected by a light-like segment if and only if $P_1$ and $P_2$ are disjoint but at distance $0$, and oriented in a compatible way in the sense that the half-space bounded by one is contained in the half-space bounded by the other.
\item $P_1^*$ and $P_2^*$ are connected by a time-like geodesic segment $s$ if and only if $P_1$ and $P_2$ are at finite distance and oriented in a compatible way, and the length of $s$ is then the hyperbolic distance between $P_1$ and $P_2$. 
\end{itemize}

If $S\subset \HH^3$ is an oriented, locally strongly convex surface (that is, its second fundamental form is positive definite at every point), one can define the dual surface $S^*\subset dS^3$ as the set of points in $dS^3$ dual to the tangent planes of $S$. A key property is that $S^*$ is then a space-like, strongly convex surface, and that the induced metric of $S^*$ is equal -- through the identification between $S$ and $S^*$ coming from the duality -- to the third fundamental form of $S$, and conversely. 

Suppose now that $E$ is a hyperbolic end. Its universal cover $\tE$ can be identified isometrically to a domain in $\HH^3$ bounded by a concave pleated surface $\partial_0\tE\subset \HH^3$. The set of points in $dS^3$ dual to the oriented planes bounding half-spaces contained in $\tE$ constitutes a {\em domain of dependence}, denoted by $\tE^*$, in $dS^3$. It is a convex domain bounded by a locally convex surface $\partial_0 \tE^*\subset dS^3$ which is dual to $\partial_0 \tE$. In a sense which can be made precise, $\partial_0 \tE^*$ is everywhere light-like, except along a real tree which is dual to the lift to $\partial_0 \tE$ of the measured bending lamination on the concave pleated surface $\partial_0E$ in $E$.

By construction, $\tE^*$ is invariant under an isometric action of the fundamental group $\Gamma$ of $E$, which is the fundamental group of the boundary component of $M$ corresponding to $E$. The quotient $E^*=\tE^*/\Gamma$ is a globally hyperbolic maximal compact (GHMC) de Sitter spacetime, with initial singularity the quotient by $\Gamma$ of the real tree described above.

Given $K\in (-1,0)$ and a surface $S$ of constant curvature $K$ in $E$, $S$ lifts to a constant curvature $K$ surface $\tS$ in $\tE$, which is locally strongly convex. The dual surface $\tS^*\subset dS^3$ is a space-like, strongly convex surface. It is also invariant under $\Gamma$, and the quotient $S^*$ is a Cauchy surface in $E^*$. Moreover, as already mentioned, $S^*$ has induced metric equal (through the identification between $S$ and $S^*$ through the duality map) to the third fundamental form of $S$, and its curvature is constant and equal to $K^*=K/(K+1)$.

In this manner, the surfaces dual in $E^*$ to the closed surfaces of constant curvature $K\in (-1,0)$ which foliate $E$ define a foliation of $E^*$ by closed surface of constant curvature $K^*\in (-\infty, -1)$. The existence and uniqueness of this foliation was proved in \cite{BBZ2}.
 
\subsection{The third fundamental form and bending lamination}

As mentioned above, we will need a slight refinement of the following statement, similar to (slightly more elaborate) results obtained in \cite[Section 6]{cyclic}.

\begin{lemma} \label{lm:III}
  Let $g$ be a fixed convex co-compact hyperbolic metric on $M$. Then $\lim_{K\to -1}\psi_K(g)=\psi(g)$ in the topology of the marked length spectrum, that is, for any closed curve $c$ on $\partial \bar M$, $\lim_{K\to -1}L_{\III_K}(c) = i(l, c)$, where $l$ is the measured pleating lamination on the boundary of the convex core of $M$.
\end{lemma}

\begin{proof}
  The proof follows directly, through the duality between hyperbolic and de Sitter space, from the main result in \cite{belraouti:asymptotic}, see below.
\end{proof}

The refinement we need is that the convergence is actually uniform on compact subsets of $\cCC_M$. While the proof is basically the same as for Lemma \ref{lm:III}, it is necessary to keep track more carefully of the convergence.

\begin{lemma} \label{lm:convergence2}
  Let $c$ be a closed curve on $\partial \bar M$. Then $L_{\III_K}(c)\to i(l,c)$ uniformly on compact subsets of $\cCC_M$ as $K\to -1$.
\end{lemma}

\begin{proof}
  Notice first that if $E$ is an end of a quasifuchsian manifold and $S$ is a convex surface in $S$, then the gradient on $S$ of the distance $d:E\to \R_{\geq 0}$ to the concave pleated boundary of $E$ is bounded by $1$.

  As a consequence, there exists $K_0>-1$ and a continuous function $\Delta:[-1,K_0]\to \R_{\geq 0}$, with $\Delta(-1)=0$, such that, for $K\in [-1,K_0]$, the restriction of $d$ to the surface $S_K$ is bounded above by $\Delta(K)$. Indeed, if $d(x)\geq d_0$ at a point $x\in S_K$, then $d\geq d_0/2$ on a ball of radius $d_0/2$ around $x$ in $S_K$. Since the projection from $S_K$ to $\partial_0E$ is contracting by a factor at least $\cosh(d)$, and since the ball of radius $r$ in $S_K$ has area equal to $\pi(\cosh(r/\sqrt{|K|})-1)$, it would follow that
  $$ \area(S_K)\geq \area(\partial_0E)+(\cosh(d_0/2)^2-1)\pi(\cosh(d_0/2\sqrt{|K|})-1)~. $$
  However it follows from the Gauss-Bonnet Theorem that $\area(S_K)=\area(\partial_0E)/|K|$, leading to an upper bound on $d_0$ when $K$ is close to $-1$.

  It then follows that, for any $x\in S_K$, the tangent plane $T_xS_K$ to $S_K$ at $x$ is at distance at most $\Delta(K)$ from any support plane of $\partial_0E$, because any support plane of $\partial_0E$ separates $x$ from $\partial_0E$, so that the distance from $T_xS_K$ to the support plane is at most equal to the distance from $x$ to $\partial_0E$. (Here we consider that two hyperbolic planes are at distance $0$ if they intersect.) Through the duality between $E$ and $E^*$, this means that any point of $S_K^*$ is at time distance at most $\Delta(K)$ from the initial singularity $\partial_0E^*$ of $E^*$. (In other terms, given $x^*\in S^*_K$, the maximal time distance from $x^*$ to $\partial_0E^*$ is at most $d_0$.) 

  For the last step of the argument, we introduce the foliation $(\Sigma_\tau)_{\tau\in \R_{>0}}$ of $E^*$ by surfaces of constant cosmological time, see \cite{belraouti:asymptotic}. Recall that the cosmological time is the time distance to the initial singularity of $E^*$. It follows from the previous argument that $S^*_K$ is in the past of $\Sigma_{\Delta(K)}$ for $K$ small enough.

  A key property of convex foliations in Lorentzian spacetimes (see \cite[Prop. 3.1]{belraouti:asymptotic}) is that given a leaf of such a foliation, for instance $\Sigma_\tau$ for some $\tau>0$, the projection along the normal lines from the past of $\Sigma_\tau$ is expanding. As a consequence, if we denote by $g_\tau$ the induced metric on $\Sigma_\tau$, we have
  $$ L_{\III_K}(c) \leq L_{g_{\Delta(K)}}(c)~, $$
  while the same argument applied to the projection on $S_K$ along the normal lines of the foliation by $K$-surfaces indicates that
  $$ i(l,c) \leq L_{\III_K}(c)~. $$

  The result therefore follows from the fact that $L_{g_\tau}(c)\to i(l,c)$ uniformly on compact subsets of $\cCC_M$, since the metric $g_\tau$ can be expressed in terms of grafting along the measured lamination $l$, see \cite{belraouti:asymptotic}.
\end{proof}

\begin{proof}[Proof of Lemma \ref{lm:convergence}]
  The lemma is a restated version, in terms of the functions $\psi_K$ and $\psi$, of Lemma \ref{lm:convergence2}.
\end{proof}

\section{Analyticity}
\label{sc:analyticity}

In this section we prove the following lemma, which will be used below in two different manners: once to show that the space of hyperbolic structures with a given measured bending lamination is contractible, and another time to prove that it cannot be contractible, unless it is a point. 

For simplicity, for $l\in \cMLrea$, we denote by $\cCC_M(l)=\psi^{-1}(\{ l\})$ the set of convex co-compact hyperbolic structures on $M$ for which the measured bending lamination on the boundary of $C(M)$ is $l$. Theorem \ref{tm:main} is equivalent to the statement that, for all $l\in \cMLrea$, $\cCC_M(l)$ is reduced to one point.

\begin{lemma} \label{lm:analytic}
  Let $l\in \cMLrea$. Then $\cCC_M(l)$ is a compact, analytic subset in $\cCC_M$.
\end{lemma}

Here by ``analytic subset'' we mean that $\cCC_M(l)$ is locally the vanishing set of a real-analytic function.

\begin{proof}
  The analytic structure of $\cCC_M(l)$ essentially follows from Bonahon's construction of the shear-bend cocycle \cite{bonahon-toulouse}.

  Let $\partial_1M, \partial_2M, \ldots, \partial_nM$ be the connected components of $\partial \bar M$, which are closed, oriented surfaces of genus at least $2$. For each $i\in \{ 1,\ldots, n\}$, we denote by $\cCP_{\partial_iM}$ the space of complex projective structures on $\partial_iM$, and by $\cCP_{\partial M}=\prod_{i=1}^n\cCP_{\partial_iM}$ the space of complex projective structures on $\partial \bar M$. For each $i$, $\cCP_{\partial_iM}$ is a complex manifold, see e.g. \cite{dumas-survey}, and so is the product $\cCP_{\partial M}$. Moreover, $\cCC_M$ embeds in $\cCP_{\partial M}$ as a complex submanifold (this follows from the definition of the complex structure on $\cCC_M$ and from the fact that the holonomy map of a complex projective structure is a local biholomorphism, see e.g. \cite[\S 5.2]{dumas-survey}).

  Bonahon \cite{bonahon-toulouse} considers, in a closed, orientable surface $S$, a maximal lamination $\lambda$. He defines an open subset $\cR(\lambda)$ of the space $\cR_S$ of representations of $\pi_1S$ in $PSL(2,\C)$, consisting of those $\rho\in\cR_S$ for which there exists a $\rho$-invariant pleated surface with pleating locus contained in $\lambda$. (Such representations are said to \emph{realize} $\lambda$, see \cite[\S I.5.3]{canaryEpsteinMarden2006fundamentals}).
  
  For each $\rho\in \cR(\lambda)$, Bonahon defines a {\em shear-bend} cocycle $\Gamma_\rho\in\cH(\lambda;\C/2\pi i\Z)$ supported on $\lambda$. He shows \cite[Theorem D]{bonahon-toulouse} that this map $\rho\to \Gamma_\rho$ induces a biholomorphic homeomorphism from $\cR(\lambda)$ to a certain cone in the complex vector space $\cH(\lambda; \C/2\pi i\Z)$.

  The construction of this map goes through pleated surfaces in the hyperbolic manifold corresponding to a complex projective structure $\rho$, and the imaginary part $i\beta_\rho$ of $\Gamma_\rho$ corresponds precisely to the bending cocycle along $\lambda$ for this pleated surface. In particular, $\beta_\rho$ is positive if and only if the pleated surface corresponding to $\lambda$ is locally convex.

  We are now ready to prove that $\cCC_M(l)$ is locally the vanishing set of real-analytic functions. Let $g\in\cCC_M(l)$ be a convex co-compact metric on $M$ that realizes $l$ as its bending lamination, i.e.\ $\psi(g)=l$.
  
  For each $i$, let $l_i\coloneqq l\restr{\partial_i M}$ be the bending lamination on the component $\partial_i M$. Pick a maximal geodesic lamination $\lambda_i$ on $\partial_i M$ containing the support of $l_i$. Denote by $Z_i$ the complex projective structure induced by $g$ on $\partial_iM$. The objects described in the above paragraphs fit together as follows:
  \begin{equation}\label{eq:local_shearbend}
    \cCP_{\partial_iM}\xrightarrow{\hol}\cR_{\partial_iM}\supseteq \cR(\lambda_i)\xrightarrow{\Gamma}\cH(\lambda_i; \C/2\pi i\Z)\xrightarrow{\Imagin}\cH(\lambda_i,\R/2\pi\Z),
  \end{equation}
  where the holonomy map $\hol$ is a local biholomorphism \cite[\S 5.2]{dumas-survey}, the inclusion $\supseteq$ is open, $\Gamma$ is a biholomophic homeomorphism and the \emph{imaginary part} map $\Imagin$ is real-analytic. Note that the composition $\Imagin\circ\Gamma$ is the bending cocycle along $\lambda_i$ and the image $\beta_i\coloneqq\beta(\hol(Z_i))$ corresponds precisely to the bending lamination $l_i$. From (\ref{eq:local_shearbend}), there is a neighborhood $U_i\subset\cCP_{\partial_iM}$ of $Z_i$ on which the composition (\ref{eq:local_shearbend}) provides a real-analytic map $B_i\colon U_i\rightarrow\cH(\lambda_i,\R/2\pi\Z)$. In particular $B_i^{-1}(\beta_i)$ is a real-analytic subset of $U_i$ that contains $Z_i$.

  All in all, we obtain a neighborhood $U\coloneqq U_1\times\cdots\times U_n$ of $g$ in $\cCP_{\partial M}$ in which the fibre $\cCC_M(l)\cap U$ is exactly given by the intersection of the complex submanifold $\cCC_M$ with the real-analytic subsets $B_i^{-1}(\beta_i)$. This proves that $\cCC_M(l)$ is a real-analytic subset of $\cCC_M$.

  The compactness of $\cCC_M(l)$ is an immediate consequence, in the more general case of geometrically finite hyperbolic manifolds, of the main result from \cite{lecuire:properness}, the properness of the bending map. (We believe that it also follows, in the convex co-compact case, from the arguments used by Bonahon and Otal to prove their Closing Lemma \cite[Prop. 8]{bonahon-otal}.) 
\end{proof}

It can be noted that one needs to use locally defined functions only in order to have an argument that applies also to manifolds with compressible boundary, since in this case some maximal laminations on the boundary do not contain the pleating locus of a pleated surface for some representations.

\section{Contractibility}

In this section we will show Lemma \ref{lm:contractible}. According to Lemma \ref{lm:III}, for a fixed convex co-compact metric on $M$, as $K\to -1$, the third fundamental forms $\III_K$ on the $K$-surfaces in the ends of $M$ converge, in the topology of the marked length spectrum, to the bending lamination. This can be interpreted heuristically as a sort of convergence of the homeomorphisms $\psi_K$ defined above to $\psi$, and our goal is indeed to show that $\psi$ is a limit of homeomorphisms. So an additional construction is necessary to identify in some way the spaces $\cT^{K^*}_{S}$ of metrics of constant curvature $K^*=K/(K+1)$, as $K\to -1$, with $\cML_S$. 

We provide one argument using earthquakes. A second possible approach would be to use geodesic currents as the natural framework in which identifying the spaces $\cT^{K^*}_{S}$ and $\cML_S$. Yet another possible approach might follow the arguments appearing in \cite[Prop. 4.4]{mazzoli2019-1}.

\subsection{Earthquakes}\label{sc:earthquakes}

In this section, we fix a closed surface $S$, which will be a connected component of $\partial \bar M$. 

We consider in this section a fixed metric $m_0\in \cT_S$. We then consider the left earthquake map based at $m_0$, $E_L^{m_0}:\cML_S\to \cT_S$. Recall that $E_L^{m_0}$ is a homeomorphism, by Thurston's Earthquake Theorem \cite{thurston-earthquakes}.

To simplify notations, we denote by $\cT^{K^*}_S$ the space of metrics of constant curvature $K^*=K/(K+1)$ on $S$ considered up to isotopy. This space $\cT^{K^*}_S$ is canonically identified with $\cT_S$, because if $m$ is a hyperbolic metric on $S$, then $(1/|K^*|)m$ is a metric of constant curvature $K^*$. We use this notation since the third fundamental form on a surface of constant curvature $K$ in $M$ has constant curvature $K^*$, see \cite{L5}.

\begin{definition}
  We denote by $u_K:\cT^{K^*}_S\to \cML_S$ the function defined as follows. Let $h\in \cT^{K^*}_S$ be a metric of constant curvature $K^*$. Let $m=|K^*|h$ be the hyperbolic metric homothetic to $h$. We set
  $$ u_K(h) = (1/\sqrt{|K^*|})(E_L^{m_0})^{-1}(m)~. $$
\end{definition}

\begin{lemma} \label{lm:earthquakes}
  \begin{enumerate}
  \item For all $K\in (-1,0)$, $u_K$ is a homeomorphism.
  \item Let $(K_n)_{n\in N}$ be a sequence in $(-1,0)$ converging to $-1$, and let $(h_n)_{n\in \N}$ be a sequence of metrics with $h_n\in \cT^{K_{n}^*}_S$ such that $h_n\to l\in \cML$ in the topology of the marked length spectrum. Then $u_{K_n}(h_n)\to l$.
  \end{enumerate}
\end{lemma}

In fact, point (2) of this lemma is not entirely sufficient for the proof of the main result, and we will need the following slightly refined statement, whose proof follows from the same arguments, see below.

\begin{lemma} \label{lm:earthquakes-uniform}
  Let $(K_n)_{n\in N}$ be a sequence in $(-1,0)$ converging to $-1$, and let $(h_n(t))_{n\in \N}$ be a sequence of families of metrics, with $h_n(t)\in \cT^{K_{n}^*}_S$, depending on a parameter $t$ in a topological space $T$. Let $\gamma$ be a closed curve on $S$. Assume that $L_{h_n(t)}(\gamma)\to l(t)\in \cML_S$ uniformly in $t$ on compact subsets of $T$, where $l(t)\in \cML_S$ for each $t\in T$. Then $i(u_{K_n}(h_n(t)), \gamma)\to i(l(t),\gamma)$ uniformly in $t$ on compact subsets of $T$.
\end{lemma}

The proof will rely on an elementary estimate, which might be of independent interest. See \cite[Lemma 16]{bonahon1992earthquakes} for a \emph{complementary} estimate.

\begin{lemma} \label{lm:estimate}
  Let $(S,h)$ be a hyperbolic surface, let $l\in \cML_S$, and let $\gamma$ be a homotopy class of simple closed curves on $S$. Then
  \[i(\gamma,l)-L_m(\gamma) \leq L_{E_L^m(l)}(\gamma) \leq i(\gamma,l)+L_m(\gamma)~,\]
  where $i(\gamma,l)$ denotes the intersection between $\gamma$ and $l$, and $L_m(\gamma)$ denotes the length of the geodesic representative of $\gamma$ for $m$.  
\end{lemma}

\begin{figure}[ht!]
\centering
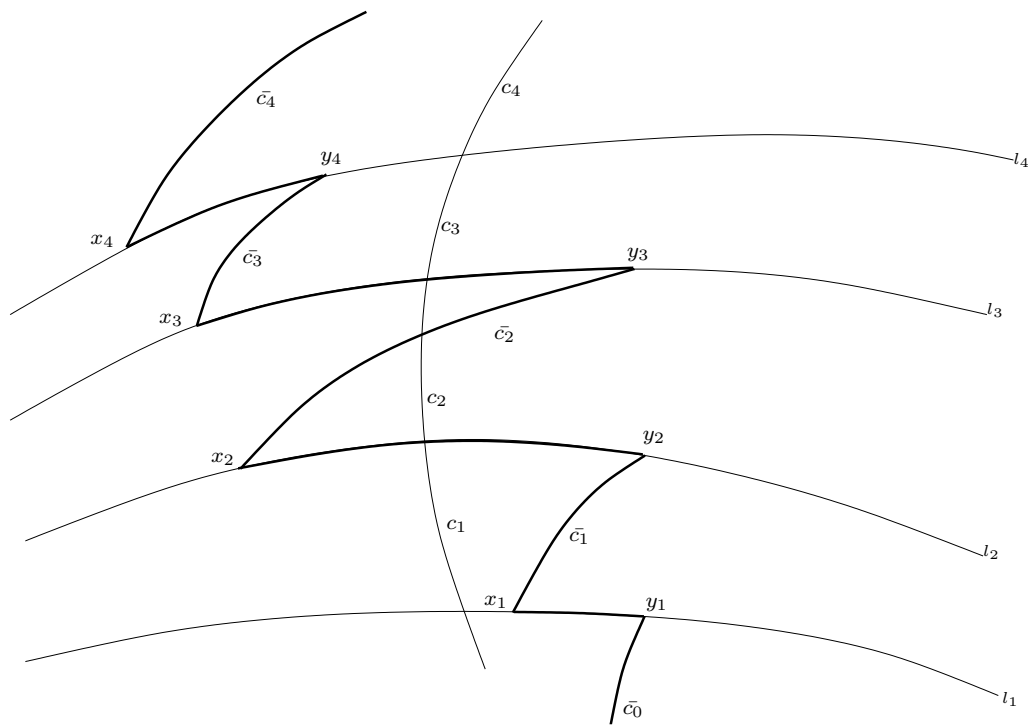
\caption{Proof of Lemma \ref{lm:estimate}, in $(\bar S, \bar m)$} \label{fg:1}
\end{figure}

\begin{proof}
  We denote by $\bar S$ the cover of $S$ with fundamental group generated by $\gamma$. We still denote by $m$ the lift of $m$ to $\bar S$, so that $(\bar S, m)$ is a hyperbolic cylinder. The proof will be done entirely in $\bar S$. We denote by $\bar c'$ the geodesic representative of $\gamma$ in $(\bar S, m)$.

  Still denote by $l$ the measured lamination on $\bar S$ corresponding to $l$. That is, the support of this lamination $l$ is the union of lifts of the geodesics in the support of $l$ on $S$ which intersect $\gamma$, with the corresponding weight. 

  We now assume that $l$ is rational in $S$ -- that is, supported on a finite set of disjoint closed curves. It is sufficient to prove the lemma in this rational case, since the general case then follows by density of rational laminations in $\cML_S$.  

  The set of geodesics in the support of $l$ which intersect $\bar c'$ in $(\bar S, m)$ is a finite set of disjoint geodesics $l_1, \ldots, l_n$ which all intersect $\bar c'$ exactly once. We assume that the $l_i$ intersect $\bar c'$ in this cyclic order.

  Denote by $\bar m=E^L_m(l)$ the image of $m$ -- the metric on $\bar S$ by the left earthquake along $l$. We denote by $c$ the geodesic representative of $\gamma$ in $(\bar S, \bar m)$, and by $\bar c$ the image of the geodesic representative $\bar c'$ under the left earthquake along $l$. That is, $\bar c$ is a disjoint union of geodesic segments $\bar c_1, \bar c_2, \ldots, \bar c_n$, with $\bar c_i$ having one endpoint on $l_i$ and one on $l_{i+1}$. Similarly, we denote by $c_i$ the segment of $c$ between its intersections with $l_i$ and $l_{i+1}$, see Figure \ref{fg:1}. We set
  $$ \gamma_i = L_{\bar m}(c_i)~, ~~\bar\gamma_i = L_{\bar m}(\bar c_i)~. $$

  We consider the geodesics $l_i, 1\leq i\leq n$ as oriented (from the left to the right side of $c$), and denote by $x_i$ the oriented distance along $l_i$ between its intersection with $c$ and its intersection with $\bar c_i$, and by $y_i$ the oriented distance along $l_i$ between its intersection with $c$ and with $\bar c_{i_1}$. By definition of a left earthquake, $y_i-x_i$ is the weight of $l_i$ in the measured lamination $l$. As a consequence,
  $$ i(\gamma, l) = \sum_{i=1}^n y_i-x_i~. $$
  The proof of the upper bound in the lemma directly follows, since there is a curve in $(\bar S,\bar m)$, in the homotopy class $\gamma$, composed of the $\bar c_i$ and of segments of the $l_i$ of length $y_i-x_i$, so that
  $$ L_{\bar m}(\gamma) = L_{\bar m}(c) \leq \sum_i L_{\bar m}(\bar c_i) + \sum_i (y_i-x_i) = L_{m}(\gamma) + i(\gamma, l)~. $$  

  We now aim at providing a lower bound on the lengths $\gamma_i$ of the geodesic segments $c_i$, $1\leq i\leq n$. For this we consider three cases, as seen on Figure \ref{fg:1}.
  \begin{itemize}
  \item $x_i\geq 0, y_{i+1}\geq 0$ (as for $i=1$ in the figure). It then follows from the triangle inequality (or its variant for 4-gons) that
    $$ \gamma_i \geq y_{i+1} - x_i - \bar\gamma_i~. $$
  \item $x_i\leq 0, y_{i+1}\geq 0$ (as for $i=2$ in the figure). The triangle inequality, applied to each of the two triangles appearing in the figure, shows that
    $$ \gamma_i \geq y_{i+1} - x_i - \bar\gamma_i~. $$
  \item $x_i\leq 0, y_{i+1}\leq 0$ (as for $i=3$ in the figure). Again the triangle inequality, applied in a slightly different way compared to the first case, still shows that
    $$ \gamma_i \geq y_{i+1} - x_i - \bar\gamma_i~. $$
  \end{itemize}
  We can now sum the inequalities obtained for $i=1,\ldots, n$ and obtain that
  $$ \sum_i \gamma_i \geq \sum_i (y_i-x_i) - \sum_i \bar\gamma_i, $$
  so that
  $$ L_{\bar m}(\gamma) \geq i(l,\gamma)- L_m(\gamma)~, $$
  as required.
\end{proof}

\begin{proof}[Proof of Lemma \ref{lm:earthquakes}]
  The first point is a direct consequence of the definition of $u_K$ and of the Earthquake Theorem.

  To prove the second point, we will use Lemma \ref{lm:estimate}. Let $m_n$ be the hyperbolic metric homothetic to $h_n$. Let $l_n\in \cML_S$ be the unique measured lamination such that $E^L_{m_0}(l_n)=m_n$, so that, by definition of $u_K$,
  $$ u_{K_n}(h_n) = \frac 1{\sqrt{|K_n^*|}}l_n~. $$
  It follows from Lemma \ref{lm:estimate} that, for all $n$,
  \[ L_{m_n}(c) - L_{m_0}(c) \leq i(l_n,c) \leq L_{m_n}(c) + L_{m_0}(c)~. \]
  This can be written, using the definition of $m_n$, as
  \[ \sqrt{|K_n^*|}L_{h_n}(c) - L_{m_0}(c)\leq i(l_n, c)\leq \sqrt{|K_n^*|}L_{h_n}(c) + L_{m_0}(c)~, \]
  and, dividing by $\sqrt{|K_n^*|}$,
  \[ L_{h_n}(c) - \frac{L_{m_0}(c)}{\sqrt{|K_n^*|}}\leq \frac{i(l_n, c)}{\sqrt{|K_n^*|}}\leq L_{h_n}(c) + \frac{L_{m_0}(c)}{\sqrt{|K_n^*|}}~. \]
  Since $L_{h_n}(c)\to i(l,c)$ by the hypothesis of the lemma, we can conclude that $\frac{i(l_n, c)}{\sqrt{|K_n^*|}}\to i(l,c)$, as required.
\end{proof}

\begin{proof}[Proof of Lemma \ref{lm:earthquakes-uniform}]
  We now consider a sequence of metrics $(h_n(t))_{n\in \N}$ of constant curvature $K^*_n$, depending on a parameter $t\in T$. As above we call $m_n(t)$ the hyperbolic metric in the same conformal class as $h_n(t)$, that is, 
  $$ m_n(t) = |K^*_n|h_n(t)~, $$
  so that by the hypothesis of the lemma,
  $$ \frac 1{\sqrt{|K^*|}}L_{m_n(t)}(\gamma) \to i(l(t),\gamma) $$
  uniformly in $t$ on compact subsets of $T$.

  Applying Lemma \ref{lm:estimate}, we have
  $$ |L_{m_n(t)}(\gamma) - i(l_n(t), \gamma)|\leq L_{m_0}(\gamma)~, $$
  where $l_n(t)=(E_L^{m_0})^{-1}(m_n(t))$. It follows that
  $$ |\sqrt{|K_n^*|}L_{h_n(t)}(\gamma)- \sqrt{|K_n^*|}i(u_{K_n}(h_n(t)), \gamma)|\leq L_{m_0}(\gamma)~, $$
  so that
  $$ |L_{h_n(t)}(\gamma)- i(u_{K_n}(h_n(t)), \gamma)|\leq \frac{L_{m_0}(\gamma)}{\sqrt{|K_n^*|}}~. $$
  It follows directly that if $L_{h_n(t)}(\gamma)\to i(l,\gamma)$ uniformly in $t$ on compact subsets of $T$ as $K\to -1$, then $i(u_{K_n}(h_n(t)), \gamma)\to i(l,\gamma)$ uniformly in $t$ on compact subsets of $T$ as $K\to -1$.
\end{proof}

\subsection{Fibres of a limit of homeomorphisms}

\label{ssc:pseudoisotopy}
We now complete the proof that the fibres of the map $\psi\colon\cC C(M)\rightarrow\cMLrea$ are contractible, using that $\psi$ is the limit of homeomorphisms and that its fibres are compact real-analytic subsets of $\cCC_M$. The following definition is taken from \cite{daverman1986decompositions}.

\begin{definition}\label{def:cellular}
  A subset $F$ of a $n$-dimensional manifold is \emph{cellular} if any neighborhood $U$ of $F$ contains a closed $n$-ball $B$ such that $F\subset\rm{int}(B)$.
\end{definition}

In particular a cellular subset must be compact, as it is a nested intersection of compact balls. A large part of the decomposition theory of manifolds \cite{daverman1986decompositions} is interested in characterizing the fibres of a limit of homeomorphisms, culminating in the acclaimed \emph{Bing's shrinking criterion}. Similar in spirit to that criterion is the following result, which we use here. We state a slightly different form, with weaker conditions better suited to the present context, but the proof is exactly the same as in \cite{finney1967pseudo}.

\begin{theorem}[{Finney \cite[Theorem 1]{finney1967pseudo}}] \label{tm:finney}
  Let $N,N'$ be two manifolds, $d$ a metric on $N'$ (compatible with its topology), and $(f_n)_{n\in\N}$ a sequence of embeddings $N\monic N'$ (i.e.\ homeomorphisms onto their image), converging uniformly on compacts to the continuous map $f\colon N\rightarrow N'$. The compact fibres of $f$ are cellular subsets of the manifold $N$.
\end{theorem}

\begin{proof}
  Let $K\coloneqq f^{-1}(x)\subseteq N$ be a compact fibre of $f$, with $x\in N'$. Let $U$ be a neighborhood of $K$ in $N$. We need to find a closed $n$-ball $B$ contained in $U$ such that $K\subset\rm{int}(B)$.

  Let $V\subseteq N$ be an open set with $K\subset V\subset \overline{V}\subset U$ and such that $\overline{V}$ is compact. We prove the two following facts.
  \begin{enumerate}
    \item The diameter (with respect to $d$) of $f_n(K)\cup\{x\}$ tends to $0$ as $n$ tends to infinity.
    \item There exists $\epsilon_0>0$ such that $d(f_n(K),f_n(\partial\overline{V}))\geq\epsilon_0$ for all $n\in\N$.
  \end{enumerate}
  We first prove (1). Assume it does not hold, then there exists $\epsilon>0$, a strictly increasing sequence $(n_k)_{k\in\N}$ in $\N$ and points $y_k,z_k\in f_{n_k}(K)\cup\{x\}$ such that $d(y_k,z_k)\geq\epsilon$. We can assume that $y_k\neq x$ for all $k$. In particular $y_k=f_{n_k}(b_k)$ for some $b_k\in K$ and by compactness we can assume that $b_{k}\rightarrow b\in K$. Up to passing to a subsequence, we can assume that either $z_k=x$ for all $k$ or $z_k\neq x$ for all $k$. In the first case, one has $d(y_k,z_k)=d(f_{n_k}(b_k),x)\rightarrow d(f(b),x)=0$ by uniform convergence of $f_n\rightarrow f$ on $K$. In the second case, one has $z_{k}=f_{n_k}(c_{k})$ for some $c_k\in K$ and again we can assume that $c_{k}\rightarrow c\in K$, so that $d(y_k,z_k)=d(f_{n_k}(b_k),f_{n_k}(c_k))\rightarrow d(f(b),f(c))=0$ for the same reason. Both cases lead to a contradiction, which concludes the proof of (1).

  We now prove (2). Assume it does not hold, then there exists a strictly increasing sequence $(n_k)_{k\in\N}$ in $\N$ and points $b_{k}\in K$, $v_k\in\partial\overline{V}$ such that $d(f_{n_k}(b_k),f_{n_k}(v_k))<\tfrac{1}{k}$ for all $k\in\N$. By compactness of both $K$ and $\partial\overline{V}$, we can assume that $b_k\rightarrow b\in K$ and $v_k\rightarrow v\in\partial\overline{V}$. Then $d(f_{n_k}(b_k),f_{n_k}(v_k))\rightarrow d(f(b),f(v))=0$ by uniform convergence of $f_n\rightarrow f$ on $\overline{V}$ and by construction. Thus $x=f(b)=f(v)$, hence $v\in f^{-1}(x)=K$, which is a contradiction since $\partial\overline{V}\cap K=\emptyset$. This concludes the proof of (2).

  With (1) and (2) in hands, we can now conclude the proof. Let $B'\subset N'$ be the closed ball around $x$ of radius $\tfrac{\epsilon_0}{3}$, where $\epsilon_0$ is the one given by (2). By (1), there exists $n_0\in\N$ large enough such that $f_{n_0}(K)\subseteq\rm{int}(B')$. By (2), we have $f_{n_0}(\partial\overline{V})\cap B'=\emptyset$. But a ball meeting a compact set but not its boundary must be contained in its interior, for otherwise the ball would be disconnected. Therefore, $B'\subseteq f_{n_0}(V)$ and, since $f_{n_0}$ is a homeomorphism onto its image, the preimage $B\coloneqq f_{n_0}^{-1}(B')$ is the required closed $n$-ball around $K$ and contained in $U$. This concludes the proof that $K$ is cellular.
\end{proof}

In general, a cellular subset $F\subset N$ can be quite pathological and far from being contractible \cite[Figure 9-1]{daverman1986decompositions}. For example, the \emph{topologist's sine curve} is cellular but not path-connected. Under suitable regularity conditions, however, a cellular subset is contractible. Actually, it is the case when $F$ is a absolute neighborhood retract (ANR). This follows from Corollaries 2B and 3A in \cite[Section 14]{daverman1986decompositions}, but we give a direct proof for the reader's convenience.

\begin{lemma}\label{lm:CANR+cellular_contractible}
  Let $F$ be a cellular subset of a manifold $N$. If $F$ is an absolute neighborhood retract (ANR), then $F$ is contractible.
\end{lemma}

\begin{proof}
  Let us suppose that $F$ is an ANR cellular subset of $N$. To prove that $F$ is contractible, we show that the identity map $\rm{id}_F\colon F\rightarrow F$ is homotopic (inside $F$) to a constant map. Since $F$ is an ANR, it is a neighborhood retract in $N$, i.e.\ there is a neighborhood $U\subset N$ of $F$ such that the inclusion $i\colon F\monic U$ has a retraction $r\colon U\epic F$. Now, by cellularity of $F$, there is a closed $n$-ball $B$ such that $F\subset\rm{int}(B)\subset B\subset U$. Since $B$ is contractible, there is a homotopy $H\colon F\times [0,1]\rightarrow B$ from the inclusion $j\colon F\monic B$ to a constant map. Then the composition
  \begin{equation*}
    r\restr{B}\circ H\colon F\times [0,1]\xrightarrow{H} B\xrightarrow{r\restr{B}} F
  \end{equation*}
  is a homotopy from $\rm{id}_F=r\restr{B}\circ j$ to a constant map, as required.
\end{proof}

The last ingredient of our argument is the following lemma, which follows directly from Corollary 8A in \cite[Section 14]{daverman1986decompositions} combined with Sullivan's local characterization of real-analytic subsets \cite{sullivan1970combinatorial}.

\begin{lemma}\label{lem:compact_real_analytic_subset_is_CANR}
  A compact real-analytic subset of $\R^n$ is an ANR.
\end{lemma}

\begin{proof}
  Let $F\subset\R^n$ be a compact real-analytic subset. Being an $ANR$ is a local property for metrizable spaces \cite[Theorem 3.2]{hanner1951theoremsANR}, which $F$ is (as a subset of $\R^n$). By Sullivan's result \cite{sullivan1970combinatorial}, $F$ is locally homeomorphic to a cone over a polyhedron, hence locally an ANR. For example, this follows from \cite[Corollary 3.5]{hanner1951theoremsANR}.
\end{proof}

Combining all the above sections, we can now provide a proof of Lemma \ref{lm:contractible}.

\begin{lemma}
  The fibres of $\psi\colon\cCC_M\rightarrow\cMLrea$ are contractible.
\end{lemma}

\begin{proof}
  For each $K\in (-1,0)$, we introduce the function $U_K:\cT^K_{\partial\bar M}\to \cML_{\partial\bar M}$, defined by applying $u_K$ for each boundary component of $\partial\bar M$.

  It follows directly from Lemma \ref{lm:convergence} and from Lemma \ref{lm:earthquakes-uniform} that for any closed curve $c$ on $\partial\bar M$, $i(U_K\circ\psi_K(-),c)\to i(\psi(-), c)$ uniformly on compact subsets of $\cCC_M$ as $K\to -1$.

  Choose a finite set of closed curves $c_1,\cdots, c_n$ in $\partial \bar M$ such that the function $l\mapsto (i(c_1, l), i(c_2,l),\cdots, i(c_n,l))$ is a topological embedding of $\cML_{\partial\bar M}$ in $\R^n$. We can use those functions to define a distance $D$ on $\cML_{\partial\bar M}$, for instance by
  $$ D(l,l')=\sum_{i=1}^n |i(c_i,l')-i(c_i,l)|~. $$
  Considering this distance $D$ on $\cML_{\partial\bar M}$ it follows from the previous paragraph that $U_K\circ\psi_K$ converges uniformly to $\psi$ on compact subsets of $\cCC_M$, with each $U_K\circ\psi_K$ being a homeomorphism onto its image. It then follows from Lemma \ref{lm:analytic}, Theorem \ref{tm:finney}, Lemma \ref{lm:CANR+cellular_contractible} and Lemma \ref{lem:compact_real_analytic_subset_is_CANR} that the fibres of $\psi$ are contractible.
\end{proof}

\section{Non-contractibility}
\label{sc:non-contractibility}



We provide in this section the proof of Lemma \ref{lm:non-contractible}. It follows directly from a result of Sullivan \cite{sullivan1970combinatorial}.\footnote{This argument was mentioned by user ``Moishe Kohan'' (https://mathoverflow.net/users/39654/moishe-kohan) in a post on mathOverflow, see \texttt{https://mathoverflow.net/q/446497} (version: 2023-05-10).}

Sullivan \cite[Corollary 2]{sullivan1970combinatorial} notices that a subset $X\in \R^n$ which is defined by real-analytic equations can be triangulated, and that it is locally the cone over a polyhedron of even Euler characteristic.


The following statement follows directly. See \cite{borelhaefliger1961classe} for an alternative proof.

\begin{lemma}
    A compact $k$-dimensional real-analytic subset $F$ of $\R^n$ is a $\rm{mod}$ 2 pseudo-manifold, i.e.\ $H_k(F,\Z_2)\neq 0$. In particular, if it is contractible then it must be a point.
\end{lemma}

\begin{proof}
    As remarked by Sullivan in \cite{sullivan1970combinatorial} below Corollary 2, the sum of all $k$-dimensional simplices in a triangulation of $F$ is a $k$-cycle $\rm{mod}$ 2. Since $F$ has dimension $k$, this cycle is not a boundary, hence it provides a nonzero element of $H_k(F,\Z_2)$.
\end{proof}

The proof of Lemma \ref{lm:non-contractible} follows directly, since Lemma \ref{lm:analytic} asserts that for all $l\in \cMLrea$, $\psi^{-1}(\{ l\})$ is a compact real-analytic subset of $\cCC_M$. Since it is contractible by Lemma \ref{lm:contractible}, it is a point.


\bibliographystyle{alpha}
\bibliography{bib.bib,/home/jean-marc/Dropbox/papiers/outils/biblio}
\end{document}

%% file: figure.pdf_t
\begin{picture}(0,0)%
\includegraphics{figure.pdf}%
\end{picture}%
\setlength{\unitlength}{2072sp}%
\begingroup\makeatletter\ifx\SetFigFont\undefined%
\gdef\SetFigFont#1#2#3#4#5{%
  \reset@font\fontsize{#1}{#2pt}%
  \fontfamily{#3}\fontseries{#4}\fontshape{#5}%
  \selectfont}%
\fi\endgroup%
\begin{picture}(11907,8566)(889,-8357)
\put(8153,-8204){\makebox(0,0)[lb]{\smash{{\SetFigFont{8}{9.6}{\rmdefault}{\mddefault}{\updefault}{\color[rgb]{0,0,0}$\bar{c_0}$}%
}}}}
\put(8426,-6924){\makebox(0,0)[lb]{\smash{{\SetFigFont{8}{9.6}{\rmdefault}{\mddefault}{\updefault}{\color[rgb]{0,0,0}$y_1$}%
}}}}
\put(8393,-4937){\makebox(0,0)[lb]{\smash{{\SetFigFont{8}{9.6}{\rmdefault}{\mddefault}{\updefault}{\color[rgb]{0,0,0}$y_2$}%
}}}}
\put(8206,-2751){\makebox(0,0)[lb]{\smash{{\SetFigFont{8}{9.6}{\rmdefault}{\mddefault}{\updefault}{\color[rgb]{0,0,0}$y_3$}%
}}}}
\put(4573,-1604){\makebox(0,0)[lb]{\smash{{\SetFigFont{8}{9.6}{\rmdefault}{\mddefault}{\updefault}{\color[rgb]{0,0,0}$y_4$}%
}}}}
\put(7519,-6144){\makebox(0,0)[lb]{\smash{{\SetFigFont{8}{9.6}{\rmdefault}{\mddefault}{\updefault}{\color[rgb]{0,0,0}$\bar{c_1}$}%
}}}}
\put(6633,-3724){\makebox(0,0)[lb]{\smash{{\SetFigFont{8}{9.6}{\rmdefault}{\mddefault}{\updefault}{\color[rgb]{0,0,0}$\bar{c_2}$}%
}}}}
\put(3653,-2824){\makebox(0,0)[lb]{\smash{{\SetFigFont{8}{9.6}{\rmdefault}{\mddefault}{\updefault}{\color[rgb]{0,0,0}$\bar{c_3}$}%
}}}}
\put(3813,-924){\makebox(0,0)[lb]{\smash{{\SetFigFont{8}{9.6}{\rmdefault}{\mddefault}{\updefault}{\color[rgb]{0,0,0}$\bar{c_4}$}%
}}}}
\put(6066,-5997){\makebox(0,0)[lb]{\smash{{\SetFigFont{8}{9.6}{\rmdefault}{\mddefault}{\updefault}{\color[rgb]{0,0,0}$c_1$}%
}}}}
\put(5833,-4491){\makebox(0,0)[lb]{\smash{{\SetFigFont{8}{9.6}{\rmdefault}{\mddefault}{\updefault}{\color[rgb]{0,0,0}$c_2$}%
}}}}
\put(5999,-2437){\makebox(0,0)[lb]{\smash{{\SetFigFont{8}{9.6}{\rmdefault}{\mddefault}{\updefault}{\color[rgb]{0,0,0}$c_3$}%
}}}}
\put(6713,-797){\makebox(0,0)[lb]{\smash{{\SetFigFont{8}{9.6}{\rmdefault}{\mddefault}{\updefault}{\color[rgb]{0,0,0}$c_4$}%
}}}}
\put(6513,-6898){\makebox(0,0)[lb]{\smash{{\SetFigFont{8}{9.6}{\rmdefault}{\mddefault}{\updefault}{\color[rgb]{0,0,0}$x_1$}%
}}}}
\put(3279,-5197){\makebox(0,0)[lb]{\smash{{\SetFigFont{8}{9.6}{\rmdefault}{\mddefault}{\updefault}{\color[rgb]{0,0,0}$x_2$}%
}}}}
\put(2666,-3544){\makebox(0,0)[lb]{\smash{{\SetFigFont{8}{9.6}{\rmdefault}{\mddefault}{\updefault}{\color[rgb]{0,0,0}$x_3$}%
}}}}
\put(1853,-2597){\makebox(0,0)[lb]{\smash{{\SetFigFont{8}{9.6}{\rmdefault}{\mddefault}{\updefault}{\color[rgb]{0,0,0}$x_4$}%
}}}}
\put(12646,-8071){\makebox(0,0)[lb]{\smash{{\SetFigFont{6}{7.2}{\rmdefault}{\mddefault}{\updefault}{\color[rgb]{0,0,0}$l_1$}%
}}}}
\put(12421,-6316){\makebox(0,0)[lb]{\smash{{\SetFigFont{6}{7.2}{\rmdefault}{\mddefault}{\updefault}{\color[rgb]{0,0,0}$l_2$}%
}}}}
\put(12466,-3436){\makebox(0,0)[lb]{\smash{{\SetFigFont{6}{7.2}{\rmdefault}{\mddefault}{\updefault}{\color[rgb]{0,0,0}$l_3$}%
}}}}
\put(12781,-1636){\makebox(0,0)[lb]{\smash{{\SetFigFont{6}{7.2}{\rmdefault}{\mddefault}{\updefault}{\color[rgb]{0,0,0}$l_4$}%
}}}}
\end{picture}%